\documentclass[12pt]{article}
\usepackage{graphicx} 
\usepackage{amsmath}
\usepackage{amssymb}
\usepackage{amsthm}

\newtheorem{theorem}{Theorem}

\newtheorem{lemma}{Lemma}
\newtheorem{definition}{Definition}

\newtheorem{conjecture}{Conjecture}

\usepackage[english]{babel}
\usepackage[utf8]{inputenc}
\usepackage[left=1in,top=1in,right=1in]{geometry}
\usepackage[colorlinks=true, allcolors=blue]{hyperref}
\usepackage[inline, shortlabels]{enumitem}
\usepackage{mathtools}
\usepackage[capitalise]{cleveref}

\DeclareMathOperator{\ex}{ex}
\newcommand{\floor}[1]{\left\lfloor #1 \right\rfloor}
\newcommand{\ceil}[1]{\left\lceil #1 \right\rceil}

\title{Weighted Tur\'an densities for hypergraphs}
\author{
David James\thanks{
Department of Mathematics, Statistics and Computer Science, University of Illinois, Chicago, IL 60607. E-mail: \texttt{davidtj2@uic.edu}.
}
\and Bernard Lidick\'y\thanks{
Department of Mathematics, Iowa State University, Ames, IA 50011. E-mail: {\tt lidicky@iastate.edu}. Research of this author is supported in part by NSF FRG DMS-2152490 and DMS-2554129, the Simons Foundation TSM-00013439 and Scott Hanna Professorship.
}
\and 
Dhruv Mubayi\thanks{
Department of Mathematics, Statistics and Computer Science, University of Illinois, Chicago, IL 60607. E-mail: \texttt{mubayi@uic.edu}. Research of this author was partially supported by NSF Awards DMS-2153576 and DMS-2552740.
}
}
\date{\today}

\begin{document}

\maketitle
\begin{abstract}
Bondy and Tuza generalized the Tur\'an problem for graphs to weighted graphs. Subsequently, F\"uredi and K\"undgen gave a solution in the case of cliques.  We consider the weighted Tur\'an problem on hypergraphs and prove the first tight asymptotic result for 3-uniform hypergraphs. Specifically, we determine the weighted Tur\'an density of the generalized triangle  $F_5 = \{abc, abd, cde\}$. We also prove weighted stability results and exact results.

Many intriguing open problems remain, the most tantalizing of which is to determine the weighted Tur\'an density of the tetrahedron. Bounds for this density using flag algebras are presented.
\end{abstract}

 \section{AI Declaration Statement}
 All proofs in this paper were obtained by the authors without any use of AI tools. After the paper was written (also without any AI tools) and ready to be submitted for publication, the authors used ChatGPT-6 to check the write-up for correctness and English typos.

\section{Introduction}

Given a graph $F$, we say a graph $G$ is $F$-free if it does not contain any subgraph isomorphic to $F$. A central problem in extremal combinatorics is to determine the maximum number of edges in an $F$-free graph on $n$ vertices. This is called the \emph{Tur\'an number} of $F$, and it is denoted by $\ex(n, F)$. The asymptotic behavior is given by the \emph{Tur\'an density}, denoted $\pi(F) = \lim_{n \to \infty} \ex(n,F) / \binom{n}{2}$. Much is known in the graph setting. For example, the Tur\'an numbers of complete graphs are determined by Tur\'an's Theorem~\cite{originalTuranPaper}. Furthermore, for all non-bipartite graphs $F$, the Tur\'an densities are given by the Erd\H{o}s--Stone--Simonovits Theorem~\cite{erdosSim, erdosStone}.

A natural generalization of this problem is to consider weighted graphs, where each edge $e$ is assigned a weight $w(e)$ (and non-edges can be viewed as weight 0). We define the weight of a graph by $w(G)$ as the sum of its edge weights. In~\cite{bondytuza}, Bondy and Tuza defined the weighted Tur\'an number $\ex_{\mathbb{Z}}(n, K_m, r)$, which is the maximum weight of a graph $G$ with integer weights on $n$ vertices such that no copy of $K_m$ has weight at least $r$, and generalized Tur\'an's Theorem to prove $\ex_{\mathbb{Z}}(n, K_m, \binom{m}{2}) = \ex(n, K_m)$. Later, F\"uredi and K\"undgen~\cite{intweightedgraph} determined the Tur\'an densities for all values of $m$ and $r$. For each fixed $m$, they also show these asymptotics hold in all but finitely many values of $r$ when the weights are natural numbers, which can be viewed as a multigraph extremal problem instead. These multigraph Tur\'an numbers are denoted $\ex_{\mathbb{N}}(n, K_m, r)$.

Another generalization of Tur\'an problems is to consider $k$-uniform hypergraphs. Here, the analogous definitions and notations are used except $G$ and $F$ are now $k$-uniform hypergraphs. Compared to the graph case, much less is known. Both exact and asymptotic results are uncommon (see~\cite{keevashSurvey}). One case where such results are known is $\ex(n, F_5)$, where $F_5$ is a 3-graph on five vertices with edges $\{123, 124, 345\}$. Let $S(n)$ be the complete tripartite 3-graph with parts of size $\floor{n/3}$, $\floor{(n+1)/3}$, and $\floor{(n+2)/3}$, and let $s(n)$ be the number of edges of $S(n)$. Note $s(n) = \frac{2}{9} \binom{n}{3} + o(n^3)$. Frankl and F\"uredi~\cite{originalF5result} showed that $\ex(n, F_5) = s(n)$ for $n \geq 3000$. Keevash and Mubayi~\cite{dhruvF5Paper} improved this by showing $\ex(n, F_5) = s(n)$ for $n \geq 33$ while also producing a stability theorem. 

In this paper, we combine both generalizations. We consider Tur\'an problems on hypergraphs with nonnegative integer weights. We say that a 3-uniform hypergraph $G = (V,E)$ is a \emph{weighted 3-graph} if there is some function $w$ that assigns each edge a weight in $\mathbb{N} = \{0,1,\ldots\}$ (and non-edges are assigned weight 0). Note that any weighted 3-graph $(G, w)$ can be expressed as a weighting $w'$ of $K_n^3$, the complete 3-uniform hypergraph, where $w'$ assigns 0 when $G$ has a non-edge and copies $w(e)$ for any $e \in E(G)$. Throughout this paper, we will use this view, so we consider $w : \binom{V}{3} \to \mathbb{N}$ and $w(G) = \sum_{e \in \binom{V}{3}} w(e)$.  

Let $F$ be a nonempty 3-uniform hypergraph and $t$ be an integer. Given a weighted 3-graph $(G, w)$ (which can be expressed as some $(K_n^3, w')$), we say that $(G, w)$ is \emph{$(F,t)$-free} if for any copy of $F$ in $K_n^3$, $\sum_{e \in F} w'(e)$ is less than $t$.

\begin{definition}
    Given a nonempty 3-uniform hypergraph $F$ and integers $t \geq 1$ and $n \geq v(F) \in \mathbb{N}$, the weighted Tur\'an number $\ex_{\mathbb{N}}(n, F, t)$ is the maximum weight of a weighted 3-graph $(G,w)$ on $n$ vertices that is $(F,t)$-free.
\end{definition}

Since $F$ is nonempty and $n \geq v(F)$, every $e \in \binom{V}{3}$ is contained in some copy of $F$, so $0 \leq w(e) \leq t-1$. Thus $\ex_{\mathbb{N}}(n, F, t)$ is finite. By the standard averaging argument, $\frac{\ex_{\mathbb{N}}(n, F, t)}{\binom{n}{3}}$ is a nonincreasing function of $n$. Hence we can define a weighted Tur\'an density as well.
\begin{definition}
    Given a nonempty 3-uniform hypergraph $F$ and integer $t \geq 1$, the weighted Tur\'an density is 
    \[\pi_{\mathbb{N}}(F, t) = \lim_{n \to \infty} \frac{\ex_{\mathbb{N}}(n,F,t)}{\binom{n}{3}}.
    \]
\end{definition}
In this paper, we investigate $(F_5, t)$-free hypergraphs for all values of $t \geq 1$. Our main result is the Tur\'an densities of $(F_5, t)$-free hypergraphs for all values of $t \geq 1$. The behavior of $\ex_{\mathbb{N}}(n, F_5, t)$ depends on the value of $t$ modulo 3. As a result, all of our results for this problem will be split into three cases.
\begin{theorem}\label{thm:all3asym}
For every integer $k \geq 0$,
\[
\pi_{\mathbb{N}}(F_5, 3k+1) = \pi_{\mathbb{N}}(F_5, 3k+2) = k\quad\text{ and } \quad \pi_{\mathbb{N}}(F_5,3k+3) = k+\frac29.
\]
\end{theorem}

We further refine these results into stability theorems in all three cases. 
\begin{theorem}\label{thm:stabAll}
For every integer $k \geq 0$ and every $\varepsilon > 0$, there exists $\delta > 0$ and an integer $N$ such that for all $n > N$, if $G$ is a 3-graph on $n$ vertices with edge weights $w$ the following hold.
\begin{enumerate}
    \item If $(G,w)$ is $(F_5, 3k+1)$-free or $(F_5, 3k+2)$-free and $w(G) >  k\binom{n}{3} - \delta n^3$, then $|\{e \in \binom{V(G)}{3} : w(e) \neq k\}| < \varepsilon n^3$. \label{thm:Stab3k+13k+2}

    \item If $(G,w)$ is $(F_5, 3k+3)$-free and $w(G) > (k + 2/9) \binom{n}{3} - \delta n^3$, then the following hold. \label{thm:Stab3k+3}
    \begin{enumerate}
        \item $|\{e \in \binom{V(G)}{3} : w(e) \geq k+2 \text{ or } w(e) \leq k-1\}| < \varepsilon n^3$. \label{thm:Stab3k+3i}
        \item There is a partition of the vertex set of $G$ as $V(G) = U_1 \cup U_2 \cup U_3$, so that all but at most $\varepsilon n^3$ edges of $G$ of weight $k+1$ or greater have one point in each $U_i$.\label{thm:Stab3k+3ii}  
    \end{enumerate}
\end{enumerate}
\end{theorem}

Using these stability theorems, we then are able to prove the exact Tur\'an numbers $\ex_{\mathbb{N}}(n, F_5, t)$ for large $n$. 
\begin{theorem}\label{thm:exactAll}
For every fixed integer $k \geq 0$, there exists some $N$ such that for all $n \geq N$,  
\begin{enumerate}
    \item $\ex_{\mathbb{N}}(n, F_5, 3k+1) = k \binom{n}{3}$, \label{thm:exact3k+1}
    \item $\ex_{\mathbb{N}}(n, F_5, 3k+2) = k \binom{n}{3} + \floor{\frac{n}{3}}$, \label{thm:exact3k+2}
    \item $\ex_{\mathbb{N}}(n, F_5, 3k+3) = k \binom{n}{3} + \ex(n, F_5)  = k \binom{n}{3} + s(n)$. \label{thm:exact3k+3}
\end{enumerate}
\end{theorem}

The tetrahedron, denoted by $K_4^3$, is the complete 3-uniform hypergraph on 4 vertices.
Determining $\ex(n, K_4^3)$ is a notoriously hard problem. 
There is a construction showing $\pi(K_4^3) \geq 5/9$ and computations by flag algebras~\cite{berlinK4} show $\pi(K_4^3) \leq 0.55733$.
A Lean-verified bound~\cite{jeong2026newupperboundturan} is $\pi(K_4^3) \leq 0.557789$.
Denote by $K_4^{3-}$ the 3-graph obtained from the tetrahedron by removing one edge. 
The best known bounds in this case are $2/7 \leq \pi(K_4^{3-}) \leq 0.2865$, where the upper bound was again obtained by flag algebras~\cite{lpsdp2026}.
With the current understanding of $\pi(K_4^3)$ and $\pi(K_4^{3-})$, it is hopeless to determine $\pi_\mathbb{N}(K_4^3, t)$ for all $t \in \mathbb{N}$. 
However, we pose the following conjecture.
\begin{conjecture}
For every integer $k \geq 0$,
\[
\pi_{\mathbb{N}}(K_4^3,4k+1) = \pi_{\mathbb{N}}(K_4^3,4k+2) = k,
\]
\[
\pi_{\mathbb{N}}(K_4^3,4k+3) = k + \pi(K_4^{3-}) \quad\text{ and } \quad 
\pi_{\mathbb{N}}(K_4^3,4k+4) = k + \pi(K_4^3).
\]
\end{conjecture}
We present supporting evidence for the conjecture.
The constructions give each edge weight $k$ or $k+1$, where edges of weight $k+1$ either do not exist or form a $K_4^{3-}$-free or $K_4^3$-free 3-graph, respectively.
The matching upper bounds for $4k+1$ are straightforward to prove.
The equality $\pi_{\mathbb{N}}(K_4^3,2)=0$ follows from the fact that the edges must form a linear hypergraph. 
The cases
$\pi_{\mathbb{N}}(K_4^3,3) = \pi(K_4^{3-})$
and
$\pi_{\mathbb{N}}(K_4^3,4) = \pi(K_4^{3})$ follow from the fact that in these cases edges of weight at least 2 must form a linear hypergraph.
For the next three nontrivial cases, we used flag algebra calculations on weighted 3-graphs to obtain the following upper bounds that seem to support the conjecture:
\[
\pi_{\mathbb{N}}(K_4^3,6) \leq 1.0096342,
\quad\quad
\pi_{\mathbb{N}}(K_4^3,7) \leq 1.3014595,
\quad\quad
\pi_{\mathbb{N}}(K_4^3,8) \leq 1.5619373.
\]
Perhaps the most attractive conjecture is $\pi_{\mathbb{N}}(K_4^3,6) =1$.

\section{Preliminaries}

For a 3-graph $G$ and its vertex $v$, we denote by $d(v)$ the degree of $v$, that is, the number of edges of $G$ containing $v$.
We restate unweighted extremal results on $F_5$ that we will need. 

\begin{lemma}[Keevash and Mubayi~\cite{dhruvF5Paper}]\label{lem:dhruv}
   If $H$ is a nonempty $F_5$-free 3-graph on $n \geq 7$ vertices, then $H$ contains a set of vertices $S$ with $|S| \in \{3,5,7\}$ such that every pair of vertices in $S$ is contained in an edge of $H[S]$ and
   \[
        \sum_{x \in S}d(x)
        \leq 
        \alpha \binom{n - |S|}{2} + O(n),
        \text{ where }
        \alpha = 
        \left\{\begin{aligned}
        2/3 & \text{ if } |S|=3\\
        1 & \text{ if } |S|=5\\
        1 & \text{ if } |S|=7
        \end{aligned}
        \right.
    \]
\end{lemma}
\begin{proof}
Let $H$ be an $F_5$-free 3-graph. Let $t_3(n-3) = \frac{2}{3} \binom{n-3}{2} + O(n)$ be the number of edges in a complete tripartite graph on $n-3$ vertices. 

If there is any edge $e = x y z$ with at most $t_3(n-3)$ edges which intersect $e$ in exactly 1 vertex, then we let $S = xyz$. Then $\sum_{x \in S} d(x) \leq \frac{2}{3} \binom{n-3}{2} + O(n)$, since the number of edges that intersect $e$ in 2 or 3 vertices is at most $O(n)$ as well.

Otherwise, $e$ has at least $t_3(n-3) + 1$ edges incident to it in exactly 1 vertex. Let $L(x)$, $L(y)$, and $L(z)$ denote the link graphs of these vertices. Define an auxiliary graph $G$ with vertex set $V(H) \setminus e$ and edge set $L(x) \cup L(y) \cup L(z)$ restricted to $V(H) \setminus e$. Suppose these link graphs are not edge disjoint. Then without loss of generality, there are some distinct vertices $a,b \in V(H) \setminus e$ where $abx$ and $aby$ are both edges. This contradicts $H$ being $F_5$-free, as it then contains an $F_5$ with edges $abx$, $aby$, and $xyz$. Hence the link graphs  when restricted to $V(H) \setminus e$ must be edge disjoint.  

We can then 3-color the edges of $G$ according to the link graph containing them. Note edges in $G$ correspond to edges incident to $e$ in exactly 1 vertex. Hence $G$ has at least $t_3(n-3) + 1$ edges, which implies it contains some $K_4$.

Suppose this $K_4$ has vertices $abcd$. For the sake of contradiction, suppose the edges of triangle $abc$ have exactly 2 colors. Without loss of generality, suppose $ab, bc \in L(x)$ while $ac \in L(y)$. Then $abcxy$ will form an $F_5$ with edges $bxa$, $bxc$, and $acy$. Thus any triangle in this $K_4$ must be either rainbow or monochromatic. This in turn implies $K_4$ is either properly colored or is monochromatic. 

If $K_4$ is properly colored, let $S = abcdxyz$.  If $K_4$ is monochromatic, let $v \in \{x,y,z\}$ be the vertex that defines the color of the $K_4$. Then let $S = abcdv$. Note in both cases, every pair of vertices in $S$ has an edge containing it in $S$.

We now verify $\sum_{x \in S} d(x) \leq \binom{n - |S|}{2} + O(n)$ in both cases. Given any distinct $p,q \in V(H) \setminus S$ and distinct $r,s \in S$, we cannot have both edges $pqr$ and $pqs$ in $H$ at the same time. By our previous observations, there is some edge $rst$ for some distinct $t \in S$. Hence this would form an $F_5$ with edges $pqr$, $pqs$, $rst$, a contradiction. Thus any pair of vertices outside of $S$ is contained in at most 1 edge incident to $S$. Hence $\sum_{x \in S} d(x) \leq \binom{n - |S|}{2} + O(n)$ as claimed.
\end{proof}
We also use the previously known $F_5$ Tur\'an results.
\begin{theorem}[Frankl and F\"uredi~\cite{originalF5result}, Keevash and Mubayi~\cite{dhruvF5Paper}]\label{thm:F5asym}
$\pi(F_5) = \frac{2}{9}$. For $n \geq 33$, $\ex(n, F_5) = s(n)$. 
\end{theorem}
\begin{theorem}[Keevash and Mubayi~\cite{dhruvF5Paper}]\label{thm:F5stab}
For any $\varepsilon > 0$, there exists $\delta > 0$ and $N$ such that if $H$ is an $F_5$-free 3-graph on $n > N$ vertices with at least $(1 - \delta) s(n)$ edges, then there is a partition of the vertex set of $H$ as $V(H) = U_1 \cup U_2 \cup U_3$, so that all but at most $\varepsilon n^3$ edges of $H$ have one point in each $U_i$.
\end{theorem}

\section{Proof of Theorem~\ref{thm:all3asym}}
We prove the asymptotic result from Theorem~\ref{thm:all3asym} in this section. 

\begin{proof}
First, we prove lower bounds with three constructions. We will later show these constructions are in fact exactly optimal. For the $3k+1$ case, consider a 3-graph on $n$ vertices with every edge having weight $k$. Then every $F_5$ has weight $3k$, so $\pi_{\mathbb{N}}(F_5, 3k+1) \geq k$. For the $3k+2$ case, we use a maximum matching on $n$ vertices. We assign these edges weight $k+1$, while all other edges are assigned weight $k$. Then every $F_5$ has weight at most $3k+1$, so $\pi_{\mathbb{N}}(F_5, 3k+2) \geq k$. For the $3k+3$ case, consider an optimal $F_5$-free 3-graph. Assign the edges of this 3-graph weight $k+1$, and assign all other non-edges weight $k$. This will be $(F_5, 3k+3)$-free, as the edges of weight $k+1$ came from a 3-graph with no copies of $F_5$. By Theorem~\ref{thm:F5asym}, we get $\pi_{\mathbb{N}}(F_5, 3k+3) \geq k + 2/9$. 

We now upper bound the densities. Let $c \in \{3k+1, 3k+2, 3k+3\}$ be fixed and $(G, w)$ be any weighted 3-graph on $n$ vertices such that every copy of $F_5$ has weight less than $c$. We first consider small values of $c$. If $c = 1$, then every edge must have weight $0$, so $w(G) \leq 0$. If $c = 2$, every edge of positive weight must be disjoint and each edge has weight at most $1$. Thus $w(G) \leq O(n)$. If $c = 3$, then each edge has weight at most 2. The edges of weight 2 must be disjoint, so these edges will contribute at most $O(n)$ total weight. On the other hand, let $H$ be the hypergraph containing the edges of weight $1$. $H$ must be $F_5$-free, so by Theorem~\ref{thm:F5asym}, $e(H) \leq \frac{2}{9} \binom{n}{3} + O(n^2)$. Hence $w(G) \leq e(H) + O(n) \leq \frac{2}{9} \binom{n}{3} + O(n^2)$. If $c = 4$, then each edge has weight at most 3. The edges of weight at least 2 must be disjoint from each other, so they contribute at most $O(n)$ weight. If every edge has weight 1, they contribute $\binom{n}{3}$ weight, so overall we get $w(G) \leq \binom{n}{3} + O(n)$ in this case. 

We now prove upper bounds for the remaining values of $c$. Let $C$ be a constant such that the $O(n)$ term from Lemma~\ref{lem:dhruv} is upper bounded by $Cn$ for all positive values of $n$. Let 
\[
M_c = 49 + 56c +  C + \sum_{i=5}^{12} \ex_{\mathbb{N}}(i, F_5, c). 
\]
For $c= 3k+1$ with $k \geq 2$ we prove $w(G) \leq k \binom{n}{3} + M_{3k+1} n^2$ by induction on the number of vertices. Similarly, for $c = 3k+2$ with $k \geq 1$, we show $w(G) \leq k \binom{n}{3} + M_{3k+2} n^2$. Lastly, for $c= 3k+3$ with $k \geq 1$, we show $w(G) \leq (k + 2/9) \binom{n}{3} + M_{3k+3} n^2$. These will imply the corresponding statements for $\pi_{\mathbb N}(F_5, c)$. For the base cases, we consider $5 \leq n \leq 12$ and note the $M_c$ coefficient proves the results. 

For the induction step, suppose the bound is true for all 3-graphs on $m < n$ vertices, and then suppose $G$ has $n$ vertices. If all edges have weight at most $k$, we are done since this implies $w(G) \leq k \binom{n}{3}$. So we suppose this is not the case.

Let $H$ be an unweighted 3-graph such that $e \in H$ if and only if $w(e) \geq k+1$. Since $H$ must be $F_5$-free,  we can apply Lemma~\ref{lem:dhruv} to find some set of vertices $S = \{x_1, \dots, x_{|S|}\}$.
We will count the weight of the edges incident to $S$. Then we will remove $S$ from $G$ and continue by induction to count the weight of $G$.

We will first count the weight of edges with exactly 1 vertex in $S$. These edges must have weight at most $c- 1 - (k+1)$, as each vertex in $S$ is contained in an edge of weight at least $k+1$.
For each $i$, let $d_j(x_i)$ be the number of edges of weight $j$ incident to $x_i$ and two vertices in $G \setminus S$. Then the weight of edges incident to vertices in $S$ is 
\[
\sum_{x_i \in S} \sum_{j \geq 1} jd_j(x_i).
\] 
This sum can be rewritten in the following form.
\begin{align}\label{eq:big}
\sum_{\ell = 1}^{c-1-(k+1)} \left( \sum_{x_i \in S} \sum_{j \geq \ell} d_j(x_i) \right)
\end{align}

We upper bound $\sum_{j \geq \ell} d_j(x_i)$ for $\ell \leq k$. Suppose $x_a, x_i \in S$ are distinct vertices such that $w(x_a u v) \geq c - (k+1) - \ell$ for some distinct $u,v \in G \setminus S$. By Lemma~\ref{lem:dhruv}, there is some edge in $S$ that contains $x_i$ and $x_a$ with weight at least $k+1$. Let this edge be $x_i x_a x_b$. Then consider the copy of $F_5$ with vertices $x_i x_a x_b u v$ and edges $x_i x_a x_b$, $x_i u v$, $x_a uv$. The sum of weights is then 
\[
w(x_i x_a x_b) + w(x_i u v) + w(x_a u v) \geq (k+1) + (c - (k+1) - \ell) + w(x_i u v) = c - \ell + w(x_i u v). 
\]
Since $G$ is $(F_5, c)$-free, we have $c- \ell + w(x_i u v) < c$. Thus we must have $w(x_i u v) < \ell$ in this case. Hence for any pair of vertices $u,v \in G \setminus S$, if we want $w(x_i u v) \geq \ell$ then we must have $w(x_a u v) < c - (k+1) - \ell$ for all other $x_a \in S$. Furthermore, we cannot have $x_a \neq x_b$ such that $w(x_a u v) \geq c - (k+1) - \ell$ and $w(x_b u v) \geq c - (k+1) - \ell $. This is because there is an edge in $S$ containing both $x_a$ and $x_b$ with weight at least $k+ 1$, so this forms an $F_5$ with weights lower bounded by
\[
2 (c - (k+1) - \ell ) + (k+1) \geq 2c - 2\ell - k - 1 \geq 2c - 3k - 1 \geq  c.
\]
This is a contradiction, so we cannot have 2 such edges. This implies for each $x_i \in S$ that
\[
\sum_{j \geq \ell} d_j(x_i) \leq \binom{n-|S|}{2} - \sum_{x_a \in S \setminus \{x_i\} } \sum_{j \geq c-(k+1)-\ell} d_j(x_a).
\]
This results in 
\begin{align}\label{ellBound}
\sum_{x_i \in S} \sum_{j \geq \ell} d_j(x_i) \leq |S|\binom{n-|S|}{2} - (|S| - 1)\sum_{x_i \in S } \sum_{j \geq c-(k+1)-\ell} d_j(x_i).
\end{align}

We now upper bound \eqref{eq:big}. We break this into 3 cases. 

\textbf{Case I.}
If $c = 3k+1$, then 
\[
\sum_{\ell = 1}^{2k-1} \left( \sum_{x_i \in S} \sum_{j \geq \ell} d_j(x_i) \right) \leq \sum_{\ell = 1}^{k} \left( \sum_{x_i \in S} \sum_{j \geq \ell} d_j(x_i) \right) + (|S| - 1) \sum_{\ell = k+1}^{2k-1} \left( \sum_{x_i \in S} \sum_{j \geq \ell} d_j(x_i) \right)
\]
 Note the extra $(|S|-1)$ factor which we added to make some cancellations nicer for this analysis. When proving our later results, we will not be adding this extra $(|S|-1)$ factor. For $1 \leq \ell \leq k-1$, we will use \eqref{ellBound} to obtain 
\begin{align*}
\sum_{\ell=1}^{k-1} \left( \sum_{x_i \in S} \sum_{j \geq \ell} d_j(x_i) \right) &\leq (k-1)|S|\binom{n-|S|}{2} - (|S| - 1) \sum_{\ell=1}^{k-1} \left( \sum_{x_i \in S } \sum_{j \geq 2k-\ell} d_j(x_i) \right)\\
&= (k-1) |S| \binom{n - |S|}{2} - (|S| - 1) \sum_{\ell=k+1}^{2k-1} \left( \sum_{x_i \in S } \sum_{j \geq \ell} d_j(x_i) \right)
\end{align*}
We also have that 
\[
\sum_{x_i \in S}  \sum_{j \geq k} d_j(x_i) \leq \sum_{x_i \in S}  \sum_{j \geq 1} d_j(x_i) \leq |S|  \binom{n-|S|}{2}.
\]
Combining these gives us 
\begin{align*}
\sum_{\ell = 1}^{2k-1} \left( \sum_{x_i \in S} \sum_{j \geq \ell} d_j(x_i) \right) &\leq (k-1) |S| \binom{n - |S|}{2} - (|S| - 1) \sum_{\ell=k+1}^{2k-1} \left( \sum_{x_i \in S } \sum_{j \geq \ell} d_j(x_i) \right)\\
 & \quad\quad\quad + |S|  \binom{n-|S|}{2} + (|S| - 1) \sum_{\ell = k+1}^{2k-1} \left( \sum_{x_i \in S} \sum_{j \geq \ell} d_j(x_i) \right)\\
 &= k |S| \binom{n - |S|}{2}. 
\end{align*}

\textbf{Case II.}
If $c = 3k + 2$, then
\[
\sum_{\ell = 1}^{2k} \left( \sum_{x_i \in S} \sum_{j \geq \ell} d_j(x_i) \right) \leq \sum_{\ell = 1}^{k} \left( \sum_{x_i \in S} \sum_{j \geq \ell} d_j(x_i) \right) + (|S| - 1) \sum_{\ell = k+1}^{2k} \left( \sum_{x_i \in S} \sum_{j \geq \ell} d_j(x_i) \right).
\]
For $1 \leq \ell \leq k$, we again will use \eqref{ellBound} to obtain 
\begin{align*}
\sum_{\ell=1}^{k} \left( \sum_{x_i \in S} \sum_{j \geq \ell} d_j(x_i) \right) &\leq k|S|\binom{n-|S|}{2} - (|S| - 1) \sum_{\ell=1}^{k} \left( \sum_{x_i \in S } \sum_{j \geq 2k+1-\ell} d_j(x_i) \right)\\
&= k |S| \binom{n - |S|}{2} - (|S| - 1) \sum_{\ell=k+1}^{2k} \left( \sum_{x_i \in S } \sum_{j \geq \ell} d_j(x_i) \right)
\end{align*}
This gives us
\begin{align*}
\sum_{\ell = 1}^{2k} \left( \sum_{x_i \in S} \sum_{j \geq \ell} d_j(x_i) \right) &\leq k |S| \binom{n - |S|}{2} - (|S| - 1) \sum_{\ell=k+1}^{2k} \left( \sum_{x_i \in S } \sum_{j \geq \ell} d_j(x_i) \right)\\
& \quad\quad\quad + (|S| - 1) \sum_{\ell = k+1}^{2k} \left( \sum_{x_i \in S} \sum_{j \geq \ell} d_j(x_i) \right) \\
&= k |S| \binom{n-|S|}{2}.
\end{align*}

\textbf{Case III.}
If $c = 3k + 3$, then
\[
\sum_{\ell = 1}^{2k+1} \left( \sum_{x_i \in S} \sum_{j \geq \ell} d_j(x_i) \right) \leq \sum_{\ell = 1}^{k+1} \left( \sum_{x_i \in S} \sum_{j \geq \ell} d_j(x_i) \right) + (|S| - 1) \sum_{\ell = k+2}^{2k+1} \left( \sum_{x_i \in S} \sum_{j \geq \ell} d_j(x_i) \right).
\]
For $1 \leq \ell \leq k$, we again will use \eqref{ellBound} to obtain 
\begin{align*}
\sum_{\ell=1}^{k} \left( \sum_{x_i \in S} \sum_{j \geq \ell} d_j(x_i) \right) &\leq k|S|\binom{n-|S|}{2} - (|S| - 1) \sum_{\ell=1}^{k} \left( \sum_{x_i \in S } \sum_{j \geq 2k+2-\ell} d_j(x_i) \right)\\
&= k |S| \binom{n - |S|}{2} - (|S| - 1) \sum_{\ell=k+2}^{2k+1} \left( \sum_{x_i \in S } \sum_{j \geq \ell} d_j(x_i) \right)
\end{align*}
For $\ell = k+1$, we can use Lemma~\ref{lem:dhruv} to obtain 
\[ 
\sum_{x_i \in S} \sum_{j \geq k+1} d_j(x_i) \leq \alpha \binom{n - |S|}{2} + O(n) \leq \alpha \binom{n - |S|}{2} + Cn.
\]
Combining these gives us 
\begin{align*}
\sum_{\ell = 1}^{2k+1} \left( \sum_{x_i \in S} \sum_{j \geq \ell} d_j(x_i) \right) &\leq k |S| \binom{n - |S|}{2} - (|S| - 1) \sum_{\ell=k+2}^{2k+1} \left( \sum_{x_i \in S } \sum_{j \geq \ell} d_j(x_i) \right) + \alpha \binom{n-  |S|}{2} \\
& \quad\quad\quad + Cn + (|S| - 1) \sum_{\ell = k+2}^{2k+1} \left( \sum_{x_i \in S} \sum_{j \geq \ell} d_j(x_i) \right)\\
&= k |S| \binom{n - |S|}{2} + \frac{\alpha}{|S|} |S| \binom{n - |S|}{2} + Cn\\
&\leq \left(k + \frac{\alpha}{|S|}\right) |S|\binom{n - |S|}{2} + Cn.
\end{align*}
Note that for each possibility in Lemma~\ref{lem:dhruv}, $\frac{\alpha}{|S|} \leq \frac{2}{9}$. Hence we get that 
\[
\sum_{\ell = 1}^{2k+1} \left( \sum_{x_i \in S} \sum_{j \geq \ell} d_j(x_i) \right) \leq \left(k + \frac{2}{9}\right) |S| \binom{n-|S|}{2} + Cn.
\]

Recall that there are other edges that intersect $S$ in 2 or 3 vertices. These contribute weight at most 
\[
c\binom{|S|}{2}(n-3) + c \binom{|S|}{3} \leq 56 c n.
\]
We also know that $G \setminus S$ should be $(F_5, c)$-free. Then for $c = 3k+1$ or $c = 3k+2$, we have $w(G) \leq k |S| \binom{n-|S|}{2} + 56 cn + w(G \setminus S)$. By the induction hypothesis, 
$w(G \setminus S) \le k \binom{n-|S|}{3} + M_c (n-|S|)^2$. Consequently, 
\begin{align*}
w(G) &\leq k|S| \binom{n-|S|}{2} + k \binom{n-|S|}{3} + M_c (n-|S|)^2 + 56cn\\
&= k(\binom{n}{3} - \binom{|S|}{3} - \binom{|S|}{2}(n-|S|)) + M_c (n-|S|)^2 + 56cn \\
&< k \binom{n}{3} + M_c (n-|S|)^2 + 56cn < k \binom{n}{3} + M_c n^2,
\end{align*} so the claim holds. 
For $c = 3k+3$, we have $w(G) \leq (k + 2/9 ) |S| \binom{n-|S|}{2} + 56cn + C n + w(G \setminus S)$. By the induction hypothesis, we have $w(G \setminus S) \le (k + 2/9) \binom{n-|S|}{3} + M_c (n-|S|)^2$. Hence,
\begin{align*}
w(G) &\leq (k + 2/9)|S| \binom{n-|S|}{2} + (k + 2/9) \binom{n-|S|}{3} +  M_c (n-|S|)^2 + 56cn + Cn\\
&= (k + 2/9)(\binom{n}{3} - \binom{|S|}{3} - \binom{|S|}{2}(n-|S|)) + M_c (n-|S|)^2 + 56cn + Cn\\
&< (k + 2/9) \binom{n}{3} + M_c (n-|S|)^2 + 56cn + Cn < (k+ 2/9) \binom{n}{3} + M_c n^2,
\end{align*}
so the claim holds. By the principle of induction, we obtain the desired Tur\'an densities.
\end{proof}

\section{Stability Results}
We extend the above results into stability results by showing near maximum $(F_5, t)$-free 3-graphs mostly have edges of weight $k$.  Note that when $c = 3k+1$ or $c = 3k+2$, the above analysis did not require  Lemma~\ref{lem:dhruv}, as we could have chosen $S$ by picking any edge of weight at least $k+1$. In these cases, we can then refine the above analysis.

\subsection{Proof of Theorem~\ref{thm:stabAll}.\ref{thm:Stab3k+13k+2}}

\begin{proof}
It suffices to prove the $3k+2$ case, as any $(F_5,3k+1)$-free graph is also $(F_5,3k+2)$-free. Suppose we are given $1 > \varepsilon > 0$. Let $\delta = \frac{\varepsilon^2}{100(4k+4)^2}$, $N$ extremely large, and $\varepsilon_1 = \frac{\varepsilon}{4k+4}$. Let $G$ be a 3-graph on $n > N$ vertices with edge weights $w$ such that every copy of $F_5$ has weight less than $3k+2$ and that $w(G) > k \binom{n}{3} - \delta n^3$. Let $t_1$ be the number of edges of weight at least $k+1$, and $t_2$ be the number of edges of weight at most $k-1$. We claim $t_1 < \varepsilon_1 n^3$. For $k = 0$ this is immediate, as every edge of positive weight must be disjoint and each edge has weight at most $1$, so $t_1 \leq n/3 < \varepsilon_1 n^3$. We then consider larger values of $k$. For the sake of contradiction, suppose there are at least $\varepsilon_1 n^3$ edges of weight at least $k+1$ in $G$.  

We will now upper bound $w(G)$. Let $t = \ceil{2 \varepsilon_1 n / 9}$. We will do this algorithmically in $t$ stages. At each stage, we count the weight of the edges incident to a chosen edge $S$ and then continue by removing this edge. 

\underline{Stage 0:}
Let $G_0 = G$ and $H_0$ be an unweighted 3-graph such that $e \in H_0$ if and only if $w(e) \geq k+1$. Then the average degree of $H_0$ is at least $3 \varepsilon_1 n^2$, so there is some vertex $x_1$ in $H_0$ with degree at least $3 \varepsilon_1 n^2$. Let $S_0 = \{x_1, x_2, x_3\}$ be an edge in $H_0$ containing $x_1$.

$S_0$ itself has weight at most $3k+1$. There are at most $3 (n-3)$ edges whose intersection with $S_0$ is size 2. These edges have weight at most $2k$, so they contribute weight at most $3 (n-3) (2k)$. We then consider edges whose intersection with $S_0$ has size 1. By refining the analysis from Theorem~\ref{thm:all3asym}, we see the sum of weights here is upper bounded by 
\begin{align*}
\sum_{x \in S_0} \sum_{j \geq 1} jd_j(x) &= \sum_{\ell = 1}^{k} \left( \sum_{x \in S_0} \sum_{j \geq \ell} d_j(x) \right) + \sum_{\ell = k+1}^{2k} \left( \sum_{x \in S_0} \sum_{j \geq \ell} d_j(x) \right) \\
&\leq k |S_0| \binom{n - |S_0|}{2} - (|S_0| - 1) \sum_{\ell=k+1}^{2k} \left( \sum_{x \in S_0 } \sum_{j \geq \ell} d_j(x) \right) \\
& \quad \quad \quad + \sum_{\ell = k+1}^{2k} \left( \sum_{x \in S_0} \sum_{j \geq \ell} d_j(x) \right)\\
&=3k\binom{n-3}{2} - \sum_{\ell = k+1}^{2k} \left( \sum_{x \in S_0} \sum_{j \geq \ell} d_j(x) \right)
\end{align*}
Recall $x_1$ has degree at least $3 \varepsilon_1 n^2$ in $H_0$. There are at most $2n$ edges that contain $x_1$ whose intersection with $S_0$ is not 1, so there are at least $3 \varepsilon_1 n^2 - 2n > \varepsilon_1 n^2$ (by choice of $N$) edges incident to $x_1$ which intersect $S_0$ in exactly 1 place. Hence
\[
\sum_{\ell = k+1}^{2k} \left( \sum_{x \in S_0} \sum_{j \geq \ell} d_j(x) \right) \geq \varepsilon_1 n^2.
\]
Thus
\[ 
w(G_0) \leq 3k+1 + 3(2k)(n-3) + 3k \binom{n-3}{2} - \varepsilon_1 n^2 + w(G_0 \setminus S_0).
\]
After removing $S_0$ from $G_0$, we lose at most $\frac{3 n^2}{2}$ edges of weight at least $k+1$. Hence $G_0 \setminus S_0$ has at least $\varepsilon_1 n^3 - \frac{3 n^2}{2}$ edges of weight at least $k+1$. 

\underline{Stage $i$ when $1 \leq i \leq t-1$:}
Let $G_{i} = G_{i-1} \setminus S_{i-1}$ and $H_{i}$ be an unweighted 3-graph such that $e \in H_{i}$ if and only if $w(e) \geq k+1$ in $G_{i}$. Note that $G_{i}$ and $H_{i}$ have $n - 3i$ vertices. From the previous stage, we lost at most $\frac{3 n^2}{2}$ edges of weight at least $k+1$. Hence $G_{i}$ has at least $\varepsilon_1 n^3 - \frac{3 n^2}{2}i$ edges of weight at least $k+1$. Then the average degree of $H_{i}$ is at least $3 \varepsilon_1 n^2 - \frac{9 n}{2}i$, so there is some vertex $x_1$ in $H_{i}$ with degree at least $3\varepsilon_1 n^2 - \frac{9 n}{2}i$. Let $S_{i} = \{x_1, x_2, x_3\}$ be an edge in $H_{i}$ containing $x_1$.

$S_{i}$ itself has weight at most $3k+1$. There are at most $3 (n-3(i+1))$ edges whose intersection with $S_{i}$ has size 2. These edges have weight at most $2k$, so they contribute at most $3 (n-3(i+1)) (2k)$. We then consider edges whose intersection with $S_i$ has size 1. By refining the analysis from Theorem~\ref{thm:all3asym}, we see the sum of weights here is upper bounded by 
\begin{align*}
\sum_{x \in S_{i}} \sum_{j \geq 1} jd_j(x) &=  \sum_{\ell = 1}^{k} \left( \sum_{x \in S_{i}} \sum_{j \geq \ell} d_j(x) \right) + \sum_{\ell = k+1}^{2k} \left( \sum_{x \in S_{i}} \sum_{j \geq \ell} d_j(x) \right) \\
&\leq k |S_{i}| \binom{n - 3i - |S_{i}|}{2} - (|S_{i}| - 1) \sum_{\ell=k+1}^{2k} \left( \sum_{x \in S_{i} } \sum_{j \geq \ell} d_j(x) \right)  \\
& \quad \quad \quad + \sum_{\ell = k+1}^{2k} \left( \sum_{x \in S_{i}} \sum_{j \geq \ell} d_j(x) \right)\\
&=3k\binom{n-3(i+1)}{2} - \sum_{\ell = k+1}^{2k} \left( \sum_{x \in S_{i}} \sum_{j \geq \ell} d_j(x) \right)
\end{align*}    
Recall $x_1$ has degree at least $3\varepsilon_1 n^2 - \frac{9 n}{2}i$ in $H_{i}$. There are at most $2n$ edges that contain $x_1$ whose intersection with $S_{i}$ is not 1, so there are at least $3\varepsilon_1 n^2 - \frac{9 n}{2}i - 2n > \varepsilon_1 n^2 - \frac{9 n}{2}i$ (by choice of $N$) edges incident to $x_1$ which intersect $S_{i}$ in exactly 1 place. Hence
\[
\sum_{\ell = k+1}^{2k} \left( \sum_{x \in S_{i}} \sum_{j \geq \ell} d_j(x) \right) \geq \varepsilon_1 n^2 - \frac{9 n}{2}i.
\]
Thus
\[ 
w(G_{i}) \leq 3k+1 + 3(2k)(n-3(i+1))  + 3k\binom{n-3(i+1)}{2} - (\varepsilon_1 n^2 - \frac{9 n}{2}i) + w(G_{i} \setminus S_{i}).
\]
After removing $S_i$ from $G_i$, we lose at most $\frac{3 n^2}{2}$ edges of weight at least $k+1$. Hence $G_i \setminus S_i$ has at least $\varepsilon_1 n^3 - \frac{3 n^2}{2}(i+1)$ edges of weight at least $k+1$. 

\underline{Stage $t = \ceil{2 \varepsilon_1 n / 9}$:}
Here, we let $G_t = G_{t-1} \setminus S_{t-1}$ as usual. Then $G_t$ has $n - 3t$ vertices. In this case, we use Theorem~\ref{thm:all3asym} to bound $w(G_t) \leq k \binom{n- 3t}{3} + O(n^2)$. Combining this gives us
\begin{align*}
w(G) &\leq w(G_t) + \sum_{i=0}^{t-1} (3k+1 + 3(2k)(n-3(i+1))  + 3k\binom{n-3(i+1)}{2} - (\varepsilon_1 n^2 - \frac{9 n}{2}i) )\\
&\leq \left(k \binom{n- 3t}{3} + 3k \sum_{i=0}^{t-1} \binom{n-3(i+1)}{2} \right) + O(n^2) \\
& \quad\quad\quad + \sum_{i=0}^{t-1} \left( 3k+1 + 3(2k)(n-3(i+1))  - \varepsilon_1 n^2 + \frac{9 n}{2}i \right)\\
&\leq \left(k \sum_{i=2}^{n-3t-1} \binom{i}{2} + k \sum_{i=n-3t}^{n-1} \binom{i}{2} \right) + O(n^2) \\
& \quad \quad\quad + \sum_{i=0}^{t-1} (3k+1 + 3(2k)(n-3(i+1))  - \varepsilon_1 n^2 + \frac{9 n}{2}i) \\
& = k \binom{n}{3} - \varepsilon_1 n^2 t + \frac{9n}{2} \binom{t}{2} +   O(n^2)\\
& \leq k \binom{n}{3} - \frac{2 \varepsilon_1^2 n^3}{9} + \frac{\varepsilon_1^2 n^3}{9} + O(n^2) = k \binom{n}{3} - \frac{\varepsilon_1^2 n^3}{9} + O(n^2) < k \binom{n}{3} - \delta n^3.
\end{align*}
This is a contradiction. Hence we have fewer than $\varepsilon_1 n^3$ edges of weight at least $k+1$, so $t_1 < \varepsilon_1 n^3$.

We argue $t_2 < (\varepsilon - \varepsilon_1) n^3$. For $k = 0$, note this is immediate, as we cannot have any edges of weight $k-1$. For all other values of $k$, suppose $t_2 \geq (\varepsilon - \varepsilon_1) n^3$ for the sake of contradiction. Then we can naively upper bound the sum of weights of the edges by doing
\begin{align*}
w(G) &\leq (3k+1) t_1 + (k-1) t_2 + k \left(\binom{n}{3} - t_1 - t_2\right) \\
&\leq k \binom{n}{3} + (2k+1)t_1 - t_2 \\
&< k \binom{n}{3} + (2k+1) \varepsilon_1 n^3 + (\varepsilon_1 - \varepsilon) n^3\\
&= k \binom{n}{3} + \left((2k+2) \frac{\varepsilon}{4k+4} - \varepsilon\right) n^3\\
&= k \binom{n}{3} - \frac{\varepsilon}{2} n^3 < k \binom{n}{3} - \delta n^3.
\end{align*}
This is a contradiction, so we must have $t_2 < (\varepsilon - \varepsilon_1) n^3$. Hence $t_1 + t_2 < \varepsilon n^3$, which means the number of edges that have weight not equal to $k$ is upper bounded by $\varepsilon n^3$. 
\end{proof}

\subsection{Restricting edge weights for \texorpdfstring{$3k+3$}{3k+3}}
The $3k+3$ case is more complicated. We will first prove Theorem~\ref{thm:stabAll}.\ref{thm:Stab3k+3i}, which we will then extend to our stronger stability theorem that builds on Theorem~\ref{thm:F5stab}. Note that unlike in the previous cases, we must use Lemma~\ref{lem:dhruv} to pick a set $S$. 

\begin{proof}
Suppose we are given $1 > \varepsilon > 0$. Let $t_1$ be  the number of edges of weight $k+1$, $t_2$ be the number of edges of weight at least $k+2$, and $t_3$ be the number of edges of weight at most $k-1$. Let $\delta = \frac{\varepsilon}{8 (2k+2)}$, $N$ extremely large. Let $G$ be a 3-graph on $n > N$ vertices with edge weights $w$ such that every copy of $F_5$ has weight less than $3k+3$ and that $w(G) > (k + 2/9) \binom{n}{3} - \delta n^3$. For the sake of contradiction, suppose $t_2 \geq \frac{\varepsilon}{4(2k+2)} n^3$. If $k = 0$, all edges have weight at most 2. Furthermore, the edges of weight 2 must be pairwise disjoint. Hence $t_2 \leq n/3 < \frac{\varepsilon}{4(2k+2)} n^3$, a contradiction. 

We now consider the case when $k \geq 1$. We will now upper bound $w(G)$ algorithmically by removing a set $S$ at each stage until we can no longer find such a set $S$.

\underline{Stage 0:} Let $G_0 = G$, $n_0 = n$, and $H_0$ be an unweighted 3-graph such that $e \in H_0$ if and only if $w(e) \geq k+1$. This graph is nonempty since we have at least $\frac{\varepsilon}{4(2k+2)} n^3$ edges in here. By Lemma~\ref{lem:dhruv}, we find some set $S_0$ of vertices in $H_0$.

The total weight of edges that intersect $S_0$ in two or three vertices is $O(n_0)$. We then consider edges whose intersection with $S_0$ has size 1. By refining the analysis of Theorem~\ref{thm:all3asym}, the sum of weights here is upper bounded by 
\begin{align*}
\sum_{x \in S_0} \sum_{j \geq 1} jd_j(x) &= \sum_{\ell = 1}^{k+1} \left( \sum_{x \in S_0} \sum_{j \geq \ell} d_j(x) \right) +  \sum_{\ell = k+2}^{2k+1} \left( \sum_{x \in S_0} \sum_{j \geq \ell} d_j(x) \right) \\
&\leq k |S_0| \binom{n_0 - |S_0|}{2} - (|S_0| - 1) \sum_{\ell=k+2}^{2k+1} \left( \sum_{x \in S_0 } \sum_{j \geq \ell} d_j(x) \right) + \alpha  \binom{n_0-|S_0|}{2} \\
& \quad \quad \quad + O(n_0) + \sum_{\ell = k+2}^{2k+1} \left( \sum_{x \in S_0} \sum_{j \geq \ell} d_j(x) \right)\\
&\leq (k+ \frac{2}{9})|S_0|\binom{n_0-|S_0|}{2} + O(n_0) - (|S_0| - 2)\sum_{\ell = k+2}^{2k+1} \left( \sum_{x \in S_0} \sum_{j \geq \ell} d_j(x) \right)
\end{align*}    
Hence
\[
w(G_0) \leq  (k+ \frac{2}{9})|S_0|\binom{n_0-|S_0|}{2} - (|S_0| - 2)\sum_{\ell = k+2}^{2k+1} \left( \sum_{x \in S_0} \sum_{j \geq \ell} d_j(x) \right)+ O(n_0) + w(G_0 \setminus S_0).
\]
The number of edges of weight at least $k+2$ that intersect $S_0$ in two or three vertices is $O(n_0)$, while the number of edges of weight at least $k+2$ that intersect in 1 vertex is $\sum_{x \in S_0} \sum_{j \geq k+2} d_j(x)$. Hence the number of edges of weight at least $k+2$ that are removed after deleting $S_0$ is 
\[
O(n_0) + \sum_{x \in S_0} \sum_{j \geq k+2} d_j(x).
\]
\underline{Stage $i$ when $i \geq 1$:} Let $G_i = G_{i-1} \setminus S_{i-1}$, $n_i = |V(G_i)|$, and $H_i$ be an unweighted 3-graph such that $e \in H_i$ if and only if $w(e) \geq k+1$. If $H_i$ is nonempty and $n_i \geq 100$, by Lemma~\ref{lem:dhruv}, we find some set $S_i$ of vertices in $H_i$.

The total weight of edges that intersect $S_i$ in two or three vertices has weight $O(n_i)$. We then consider edges whose intersection with $S_i$ has size 1. By the analysis in the proof of Theorem~\ref{thm:all3asym}, the sum of weights here is upper bounded by 
\begin{align*}
\sum_{x \in S_i} \sum_{j \geq 1} jd_j(x) &= \sum_{\ell = 1}^{k+1} \left( \sum_{x \in S_i} \sum_{j \geq \ell} d_j(x) \right) +  \sum_{\ell = k+2}^{2k+1} \left( \sum_{x \in S_i} \sum_{j \geq \ell} d_j(x) \right) \\
&\leq k |S_i| \binom{n_i - |S_i|}{2} - (|S_i| - 1) \sum_{\ell=k+2}^{2k+1} \left( \sum_{x \in S_i } \sum_{j \geq \ell} d_j(x) \right) + \alpha  \binom{n_i-|S_i|}{2} \\
& \quad \quad \quad + O(n_i) + \sum_{\ell = k+2}^{2k+1} \left( \sum_{x \in S_i} \sum_{j \geq \ell} d_j(x) \right)\\
&\leq (k+ \frac{2}{9})|S_i|\binom{n_i-|S_i|}{2} + O(n_i) - (|S_i| - 2)\sum_{\ell = k+2}^{2k+1} \left( \sum_{x \in S_i} \sum_{j \geq \ell} d_j(x) \right)
\end{align*}    
Hence
\[
w(G_i) \leq  (k+ \frac{2}{9})|S_i|\binom{n_i-|S_i|}{2} - (|S_i| - 2)\sum_{\ell = k+2}^{2k+1} \left( \sum_{x \in S_i} \sum_{j \geq \ell} d_j(x) \right)+ O(n_i) + w(G_i \setminus S_i).
\]
Note that after removing $S_i$, the number of edges of weight at least $k+2$ that intersect $S_i$ in two or three vertices is $O(n_i)$, while the number of edges of weight at least $k+2$ that intersect in 1 vertex is $\sum_{x \in S_i} \sum_{j \geq k+2} d_j(x)$. Hence the number of edges of weight at least $k+2$ that are removed after deleting $S_i$ is 
\[
O(n_i) + \sum_{x \in S_i} \sum_{j \geq k+2} d_j(x).
\]
However, if $H_i$ is empty or $n_i < 100$, then we stop the algorithm by applying Theorem~\ref{thm:all3asym} to obtain $w(G_i) \leq (k + \frac{2}{9}) \binom{n_i}{3} + O(n_i^2)$.

Suppose the algorithm ends at stage $t$. Then we see
\begin{align*}
w(G) &\leq w(G_t) + \sum_{i=0}^{t-1}\left( (k+ \frac{2}{9})|S_i|\binom{n_i-|S_i|}{2} - (|S_i| - 2)\sum_{\ell = k+2}^{2k+1} \left( \sum_{x \in S_i} \sum_{j \geq \ell} d_j(x) \right)+ O(n_i) \right)\\
&\leq \left( \left(k + \frac{2}{9}\right) \binom{n_t}{3} + \left(k + \frac{2}{9}\right) \sum_{i=0}^{t-1} |S_i| \binom{n_i - |S_i|}{2} \right) + O(n_t^2)\\
&\quad\quad\quad + \sum_{i=0}^{t-1} \left(- (|S_i| - 2)\sum_{\ell = k+2}^{2k+1} \left( \sum_{x \in S_i} \sum_{j \geq \ell} d_j(x) \right)+ O(n_i)\right)\\
&\leq \left(k + \frac{2}{9}\right) \binom{n}{3} + O (n^2) - \sum_{i=0}^{t-1} (|S_i| - 2)\sum_{\ell = k+2}^{2k+1} \left( \sum_{x \in S_i} \sum_{j \geq \ell} d_j(x) \right).
\end{align*}
Note that by stage $t$, we have removed all but at most $\binom{100}{3}$ edges of weight at least $k+2$. By summing over the number of these edges we removed at each step, we see that
\[
\frac{\varepsilon}{4(2k+2)} n^3 \leq \sum_{i=0}^{t-1} \left( O(n_i) + \sum_{x \in S_i} \sum_{j \geq k+2} d_j(x) \right) + \binom{100}{3} \leq O(n^2) + \sum_{i=0}^{t-1} \sum_{x \in S_i} \sum_{j \geq k+2} d_j(x).
\]
Thus 
\[
\frac{\varepsilon}{4(2k+2)} n^3 - O(n^2) \leq \sum_{i=0}^{t-1} \sum_{x \in S_i} \sum_{j \geq k+2} d_j(x) \leq \sum_{i=0}^{t-1} (|S_i| - 2)\sum_{\ell = k+2}^{2k+1} \left( \sum_{x \in S_i} \sum_{j \geq \ell} d_j(x) \right).
\]
Applying this to our previous upper bound gives us 
\begin{align*}
w(G) &\leq \left(k + \frac{2}{9}\right) \binom{n}{3} + O (n^2) - \sum_{i=0}^{t-1} (|S_i| - 2)\sum_{\ell = k+2}^{2k+1} \left( \sum_{x \in S_i} \sum_{j \geq \ell} d_j(x) \right) \\
&\leq \left(k + \frac{2}{9}\right) \binom{n}{3} + O (n^2) - \frac{\varepsilon}{4(2k+2)} n^3 < \left(k + \frac{2}{9}\right) \binom{n}{3}  - \delta n^3.
\end{align*}
This is a contradiction. Hence $t_2 \leq \frac{\varepsilon}{4(2k+2)} n^3$.

We now consider $t_3$. We will show $t_3 < \frac{\varepsilon}{2} n^3$. For $k = 0$, $t_3 = 0$ since we permit only nonnegative weights, so we are done. For $k \geq 1$, we know 
\begin{align*}
\left(k + \frac{2}{9} \right) \binom{n}{3} - \delta n^3 &\leq w(G)\\
&\leq k \left(\binom{n}{3} - t_1 - t_2 - t_3 \right) + (k+1) t_1 + (3k+2) t_2 + (k-1) t_3 \\
&\leq k \binom{n}{3} + t_1 + (2k + 2) t_2 - t_3.
\end{align*}
Consider the hypergraph formed by the edges of weight $k+1$. This hypergraph is $F_5$-free. Hence for large enough $n$, we know $t_1 \leq  \frac{2}{9} \binom{n}{3} + \delta n^3$. Thus
\begin{align*}
\left(k + \frac{2}{9} \right) \binom{n}{3} - \delta n^3 &\leq k \binom{n}{3} + \frac{2}{9} \binom{n}{3} + \delta n^3 + (2k + 2) t_2 - t_3\\
\Rightarrow t_3 &\leq (2k+2) t_2 + 2 \delta n^3 \leq \left(\frac{\varepsilon}{4} + \frac{\varepsilon}{4 (2k+2)} \right) n^3 < \frac{\varepsilon}{2} n^3. 
\end{align*}
Thus
\begin{align*}
|\{e \in \binom{V(G)}{3} : w(e) \geq k+2 \text{ or } w(e) \leq k-1\}| = t_2 + t_3 \leq \frac{\varepsilon}{4(2k+2)} n^3 + \frac{\varepsilon}{2} n^3 < \varepsilon n^3. 
\end{align*}
\end{proof}
\subsection{Proof of Theorem~\ref{thm:stabAll}.\ref{thm:Stab3k+3}}
Combining Theorem~\ref{thm:stabAll}.\ref{thm:Stab3k+3i} with Theorem~\ref{thm:F5stab}, we obtain the stability theorem for the $3k+3$ case.

\begin{proof}
Suppose we have $0 < \varepsilon < 1$. By Theorem~\ref{thm:F5stab}, we get some $\delta_1 < \varepsilon$ and $N_1$. We have already proven Theorem~\ref{thm:stabAll}.\ref{thm:Stab3k+3i}, so we apply it to $\frac{\delta_1}{60(2k+1)}$ to get some $\delta = \frac{\delta_1}{60(2k+1)} \frac{1}{8(2k+2)}$ and $N \gg N_1$ large. Now, suppose $G$ is $(F_5, 3k+3)$-free with $w(G) > (k + 2/9) \binom{n}{3} - \delta n^3$. Let $t_1$ be the number of edges of weight $k+1$, and $t_2$ the number of edges of weight at least $k+2$. By Theorem~\ref{thm:stabAll}.\ref{thm:Stab3k+3i}, $t_2 \leq \frac{\delta_1}{60(2k+1)} n^3$. Then
\begin{align*}
\left(k  + \frac{2}{9} \right)\binom{n}{3} - \delta n^3 < w(G) &\leq k \left( \binom{n}{3} - t_1 - t_2 \right) + (k+1) t_1 + (3k+2) t_2\\
&\leq k \binom{n}{3} + t_1 + (2k + 2)t_2.
\end{align*}
Hence 
\[
t_1 + t_2 \geq \frac{2}{9} \binom{n}{3} - (2k+1) t_2 - \delta n^3 \geq \frac{2}{9} \binom{n}{3} - \frac{\delta_1}{60} n^3  - \delta n^3 > \frac{2}{9} \binom{n}{3} - \frac{2\delta_1}{60} n^3 > (1 - \delta_1) s(n).
\]
Let $H$ be the 3-graph whose edges are the triples of weight at least $k+1$. Then $H$ is an $F_5$-free 3-graph on $t_1 + t_2$ edges. By Theorem~\ref{thm:F5stab}, there is a partition of the vertex set of $G$ as $V(G) = U_1 \cup U_2 \cup U_3$, so that all but at most $\varepsilon n^3$ edges of $H$ have one point in each $U_i$.  Thus all but $\varepsilon n^3$ edges of $G$ of weight $k+1$ or higher have one point in each $U_i$, so Theorem~\ref{thm:stabAll}.\ref{thm:Stab3k+3ii} holds.   
\end{proof}

\section{Exact Results}

Recall the constructions we presented in Theorem~\ref{thm:all3asym}. We now prove these constructions are in fact exactly optimal for large $n$.

\subsection{Proof of Theorem~\ref{thm:exactAll}.\ref{thm:exact3k+1}}

\begin{proof}
Let $n$ be a large integer and let $G$ be a 3-graph on $n$ vertices with edge weights $w$ that is $(F_5, 3k+1)$-free. If $k = 0$, every edge must have weight $0$, or else the edge will form a copy of $F_5$ of positive weight, so we immediately get the desired result. We now turn to the case when $k > 0$. For the sake of contradiction, suppose $w(G) > k \binom{n}{3}$. Since $n$ is very large, we can apply Theorem~\ref{thm:stabAll}.\ref{thm:Stab3k+13k+2} to $\varepsilon = \frac{1}{200k}$. Let $t_i$ be the number of edges of weight $i$, let $M$ be the set of all edges of weight at most $k-1$, and $B$ be the set of all edges of weight at least $k+1$. By stability, $|M| + |B| < \varepsilon n^3$. On the other hand, we see 
\begin{align*}
w(G) &= \sum_{i=0}^{3k} i t_i  = k \binom{n}{3} + \sum_{i=0}^{3k} (i - k) t_i = k \binom{n}{3} + \sum_{i=0}^{k-1} (i-k) t_i + \sum_{i=k+1}^{3k} (i-k) t_i \\
&\leq k \binom{n}{3} - \sum_{i=0}^{k-1} t_i + 2k \sum_{i=k+1}^{3k} t_i\\
&= k \binom{n}{3} - |M| + 2k |B|.
\end{align*}
Thus
\[
k \binom{n}{3} < k \binom{n}{3} - |M| + 2k |B| \Longrightarrow |M| < 2k |B|.
\]
Therefore $B$ is nonempty for the above inequality to hold. Given a vertex $v$, let $N_B(v)$ denote the set of edges of $B$ that are incident to $v$ and $d_B(v) = |N_B(v)|$. Similarly, let $N_M(v)$ denote the set of edges of $M$ that are incident to $v$ and $d_M(v) = |N_M(v)|$. Suppose $e = xyz \in B$. Fix some pair of distinct vertices $a,b$ such that $a,b \not\in e$. We claim at least 2 of $abx, aby, abz$ are in $M$. For the sake of contradiction, suppose $abx$ and $aby$ are both not in $M$. Then $abx$ and $aby$ both have weight at least $k$, while $xyz$ has weight at least $k+1$. This means $abxyz$ is an $F_5$ with weight $3k+1$, a contradiction. Hence at least 2 of these edges are in $M$. Thus there are at least $2\binom{n-3}{2}$ edges incident to $e$ that are in $M$ whose intersection with $e$ is exactly 1. By averaging, there is a vertex in $e$ incident to at least $\frac{2}{3} \binom{n-3}{2}$ edges of $M$. Hence for any $e \in B$, there is some vertex $v \in e$ such that $d_M(v) \geq \frac{2}{3} \binom{n-3}{2}$.

Let $A = \{v \in V(G) : d_M(v) \geq  \frac{2}{3} \binom{n-3}{2}\}$. Since every edge $e \in B$ has a vertex $v \in A$, we see that 
\[
\sum_{v \in A} d_B(v) \geq |B| > \frac{|M|}{2k} = \frac{1}{6k} \sum_{v \in V} d_M(v) \geq \frac{1}{6k} \sum_{v \in A} d_M(v) \geq \frac{|A|}{9k} \binom{n-3}{2}.
\]
By averaging, there is some vertex $u \in A$ such that $d_B(u) \geq \frac{1}{9k} \binom{n-3}{2}$. Let $e = xyu \in B$ be any edge incident to $u$. Then by our previous argument, for any distinct vertices $a,b \in V(G) \setminus e$, at least 2 of $abx$, $aby$, or $abu$ belong to $M$. This means at least one of $abx$ or $aby \in M$. Thus each of these edges $e \in B$ is incident with at least $\binom{n-3}{2}$ edges $e' \in M$ such that $|e \cap e'| = 1$ and $u \not\in e \cap e'$. Note that each $e'$ is incident to at most $3n$ edges $e \in B$ that are incident to $u$. Thus
\[
|M| \geq \frac{\sum_{e \in N_B(u)} \binom{n-3}{2}}{3n} \geq \frac{\binom{n-3}{2}^2}{27nk} = \frac{(n-3)^2 (n-4)^2}{108 nk} > \varepsilon n^3.
\]
This contradicts $|M| + |B| < \varepsilon n^3$. Hence for large $n$, $w(G) \leq k \binom{n}{3}$, so $\ex_{\mathbb{N}}(n, F_5, 3k+1) = k \binom{n}{3}$.
\end{proof}

\subsection{Proof of Theorem~\ref{thm:exactAll}.\ref{thm:exact3k+2}}
For the $3k+2$ case, we show the exact result by precisely analyzing the recurrence. 
\begin{lemma}~\label{lem:3k+1rec}
If $n \geq 8$, \[\ex_{\mathbb{N}}(n, F_5, 3k+2) \leq \max \begin{cases}
k \binom{n}{3}, \\
k\binom{n-1}{2} + k\binom{n-2}{2} + k\binom{n-3}{2} + 1 + \ex_{\mathbb{N}}(n-3, F_5, 3k+2)
\end{cases}\]
\end{lemma}
\begin{proof}
Let $G$ be a $3$-graph on $n$ vertices with weights $w$ that is $(F_5, 3k+2)$-free such that $w(G) = \ex_{\mathbb{N}}(n, F_5, 3k+2)$. If $G$ has no edges of weight $k+1$ or greater, then $w(G) \leq k \binom{n}{3}$. Otherwise, let $S = \{x,y,z\}$ be any edge of weight at least $k+1$. 

We will upper bound $w(G)$ by looking at the size of intersections with $S$. Let $w_0$, $w_1$, $w_2$, and $w_3$ represent the total weight of edges whose intersection with $S$ has size $0$, $1$, $2$, and $3$ respectively. Let $w(S) = k + 1 + j$ for some $0 \leq j \leq 2k$, so $w_3 = k + 1 +j$. We also know $w_0 \leq \ex_{\mathbb{N}}(n-3, F_5, 3k+2)$. From our analysis in Theorem~\ref{thm:stabAll}.\ref{thm:Stab3k+13k+2}, we know
\[
w_1 \leq 3k\binom{n-3}{2} - \sum_{\ell = k+1}^{2k} \left( \sum_{x_i \in S} \sum_{r \geq \ell} d_r(x_i) \right) = 3k\binom{n-3}{2} - \sum_{\ell = k+1}^{2k} \left( \sum_{x_i \in S}  (\ell-k) d_\ell(x_i) \right).
\]
Recall that 
\begin{align*}
w_1 &= \sum_{\ell=0}^{2k} \left( \sum_{x_i \in S} \ell d_\ell(x_i)\right)\\ &= 3k \binom{n-3}{2} - \sum_{\ell = 0}^{k-1} \left( \sum_{x_i \in S}  (k-\ell) d_\ell(x_i) \right) + \sum_{\ell = k+1}^{2k} \left( \sum_{x_i \in S}  (\ell - k) d_\ell(x_i) \right).
\end{align*}
Using our previous upper bound, this implies that 
\[
2\sum_{\ell = k+1}^{2k} \left( \sum_{x_i \in S}  (\ell-k) d_\ell(x_i) \right) \leq \sum_{\ell = 0}^{k-1} \left( \sum_{x_i \in S}  (k-\ell) d_\ell(x_i) \right).
\]
Applying this, we get a new upper bound
\[
w_1 \leq 3k \binom{n-3}{2} - \frac{1}{2}\sum_{\ell = 0}^{k-1} \left( \sum_{x_i \in S}  (k-\ell) d_\ell(x_i) \right).
\]
For any distinct $a,b \in V(G) \setminus S$, we have the inequalities
\begin{align*}
w(xyz) + w(xya) + w(zab) \leq 3k + 1 \Rightarrow w(xya) + w(zab) \leq 2k - j\\
w(xyz) + w(xza) + w(yab) \leq 3k + 1 \Rightarrow w(xza) + w(yab) \leq 2k - j\\
w(xyz) + w(yza) + w(xab) \leq 3k + 1 \Rightarrow w(yza) + w(xab) \leq 2k - j
\end{align*}
Consider any edge $e$ whose intersection with $S$ has size $2$. Without loss of generality, let $e = xya$. Define $i_{xya} = \max\{w(xya) - (k-j), 0\}$. By the above inequalities, we have $i_{xya} \leq k$. Then for all $b \in V(G) \setminus (S \cup \{a\})$, by the above inequalities, we have $w(xya) - (k - j) \leq k - w(zab)$ so $i_{xya} \leq \max\{k - w(zab), 0\}$. For a given $a$, there are $n-4$ choices for $b$. However, symmetrically, we get $w(xyb) - (k - j) \leq k - w(zab)$ so $i_{xyb} \leq \max\{k - w(zab), 0\}$. Hence $\frac{i_{xya} + i_{xyb}}{2} \leq \max\{i_{xya}, i_{xyb}\} \leq \max\{k - w(zab), 0\}$. Let $H$ be the set of all edges with intersection with $S$ of size 2 and weight at least $k-j + 1$. Summing yields
\[
\sum_{e \in H} (n-4) \frac{i_e}{2} \leq \sum_{\ell = 0}^{k-1} \left( \sum_{x_i \in S}  (k-\ell) d_\ell(x_i) \right).
\]
Thus 
\[
w_1 \leq 3k \binom{n-3}{2} - \sum_{e \in H} \frac{(n-4)}{4} i_e
\] and 
\[
w_2 \leq 3(k-j) (n-3) + \sum_{e \in H} i_e.
\]
Note that for $n \geq 8$, $(n-4)/4 \geq 1$. Altogether, we get 
\begin{align*}
\ex_{\mathbb{N}}(n, F_5, 3k+2) &\leq w_3 + w_2 + w_1 + w_0 \\
    &\leq (k+1+j) + \left(3(k-j) (n-3) + \sum_{e \in H} i_e\right)\\
    & \quad\quad\quad\quad+ \left(3k \binom{n-3}{2} - \sum_{e \in H} \frac{(n-4)}{4} i_e \right) + \ex_{\mathbb{N}}(n-3, F_5, 3k+2)\\
    &\leq 3k \binom{n-3}{2} + 3(k-j) (n-3) + k + 1 + j + \ex_{\mathbb{N}}(n-3, F_5, 3k+2)\\
    &= 3k \binom{n-3}{2} + 3kn -8k + 1 + 10j - 3nj + \ex_{\mathbb{N}}(n-3, F_5, 3k+2)\\
    &= k \binom{n-1}{2} + k\binom{n-2}{2} + k\binom{n-3}{2} + 1\\
    &\quad\quad\quad\quad+ j(10 - 3n) + \ex_{\mathbb{N}}(n-3, F_5, 3k+2)\\
    &\leq k \binom{n-1}{2} + k\binom{n-2}{2} + k\binom{n-3}{2} + 1 + \ex_{\mathbb{N}}(n-3, F_5, 3k+2)
\end{align*}
\end{proof}
Note this inequality is achievable when $S$ has weight $k+1$, and all edges contributing to $w_1$ and $w_2$ have weight $k$. This heavily suggests our lower bound is exactly optimal. To finish, we use a progressive induction. We say a $(F_5, 3k+2)$-free 3-graph is $(k+1)$-sparse if all edges of weight at least $k+1$ are disjoint from each other. 
\begin{lemma}
Let $T(n)$ be the maximum weight of a $(F_5, 3k+2)$-free $(k+1)$-sparse 3-graph on $n \geq 5$ vertices. We claim $T(n) = k \binom{n}{3} + \floor{\frac{n}{3}}$. 
\end{lemma}
\begin{proof}
We prove this claim by induction on $n$. For the base cases, we have $n=5,6,7$. Let $G$ be a $(F_5, 3k+2)$-free $(k+1)$-sparse 3-graph on $n$ vertices. First, let $n = 5$. If $G$ has no edges of weight $k+1$ or higher, then $w(G) \leq k \binom{5}{3} = 10k$. Otherwise, $G$ has exactly 1 edge of weight $k+1$ or higher. Without loss of generality, let $V(G) = \{1,2,3,4,5\}$ and suppose edge $123$ has weight at least $k+1$. Note that the triples $123$, $145$, and $245$ form a copy of $F_5$. Hence
\[
w(123) + w(145) + w(245) \leq 3k + 1
\]
Since $G$ is $(k+1)$-sparse, the seven edges other than $123$, $145$, $245$ must have weight at most $k$. Thus 
\[
w(G) \leq w(123) + w(145) + w(245) + 7k \leq 10k + 1 = k \binom{5}{3} + \floor{\frac{5}{3}}.
\]
This is achievable by making $w(123) = k+1$ and all other edges having weight $k$. 

Next, let $n = 6$. If $G$ has no edges of weight $k+1$ or higher, then $w(G) \leq k \binom{6}{3} = 20k$. If $G$ has exactly 1 edge of weight $k+1$ or higher, by the same argument as before we will have $w(G) \leq (3k+1) + 17k = 20 k + 1$. Otherwise, $G$ has 2 edges of weight $k+1$ or higher. Since $G$ is $(k+1)$-sparse, these edges are disjoint. Without loss of generality, let $V(G) = \{1,2,3,4,5,6\}$ and suppose edges $123$ and $456$ have weight at least $k+1$. Note that the triples $123$, $145$, and $245$ form a copy of $F_5$, so
\[
w(123) + w(145) + w(245) \leq 3k + 1.
\]
Similarly, the triples $456$, $124$, and $125$ form a copy of $F_5$ that is edge-disjoint from the previous copy of $F_5$. Thus 
\[
w(456) + w(124) + w(125) \leq 3k + 1.
\] Since $G$ is $(k+1)$-sparse, all other edges have weight at most $k$. Hence
\[
w(G) \leq (w(123) + w(145) + w(245)) + (w(456) + w(124) + w(125)) + 14k \leq k \binom{6}{3} + \floor{\frac{6}{3}}.
\]
This is achievable by making $w(123) = w(456) = k+1$ and all other edges having weight $k$. 

For the last base case, let $n = 7$. If $G$ has no edges of weight $k+1$ or higher, then $w(G) \leq k \binom{7}{3} = 35k$. If $G$ has exactly 1 edge of weight $k+1$ or higher, by the same argument as before we will have $w(G) \leq (3k+1) + 32k = 35 k + 1$. Otherwise, $G$ has 2 edges of weight $k+1$ or higher. Since $G$ is $(k+1)$-sparse, these edges are disjoint. Without loss of generality, let $V(G) = \{1,2,3,4,5,6,7\}$ and suppose edges $123$ and $456$ have weight at least $k+1$. By the same arguments as before, we know that 
\[
w(123) + w(145) + w(245) \leq 3k + 1.
\]
and 
\[
w(456) + w(124) + w(125) \leq 3k + 1.
\] All 29 other edges have weight at most $k$. Hence
\[
w(G) \leq (w(123) + w(145) + w(245)) + (w(456) + w(124) + w(125)) + 29k \leq k \binom{7}{3} + \floor{\frac{7}{3}}.
\]
This is achievable by making $w(123) = w(456) = k+1$ and all other edges having weight $k$. Hence $T(n) = k \binom{n}{3} + \floor{\frac{n}{3}}$ for $n=5$, $6$, and $7$.

For the induction step, suppose we have $n \geq 8$. Applying the proof of  Lemma~\ref{lem:3k+1rec} with $w_0 \leq T(n-3)$, we get $T(n) \leq k\binom{n-1}{2} + k\binom{n-2}{2} + k\binom{n-3}{2} + 1 + T(n-3)$. By the induction hypothesis, we know $T(n-3) = k \binom{n-3}{3} + \floor{\frac{n-3}{3}}$. Then 
\begin{align*}
T(n) &\leq \left( k\binom{n-1}{2} + k\binom{n-2}{2} + k\binom{n-3}{2} + k\binom{n-3}{3} \right) + \left( 1 + \floor{\frac{n-3}{3}} \right)\\
&= k \binom{n}{3} + \floor{\frac{n}{3}}.
\end{align*}
This is achievable by our previous lower bound construction, which takes a maximum matching on $n$ vertices and gives these edges weight $k+1$ while making all other edges have weight $k$. Hence $T(n) = k \binom{n}{3} + \floor{\frac{n}{3}}$ for all $n \geq 5$ by induction.
\end{proof}

By definition, $\ex_{\mathbb{N}}(n, F_5, 3k+2) \geq T(n)$. We define the non-negative function $h(n)$ such that $\ex_{\mathbb{N}}(n, F_5, 3k+2) = T(n) + h(n)$. 
\begin{lemma}
    For $n \geq 12$, $h(n) \leq h(n-3) - 1$ or $h(n) = 0$.
\end{lemma}
\begin{proof}
Let $G$ be a $(F_5, 3k+2)$-free 3-graph with weight $\ex_{\mathbb{N}}(n, F_5, 3k+2)$. If $G$ is $(k+1)$-sparse, then $\ex_{\mathbb{N}}(n, F_5, 3k+2) = T(n)$, so $h(n) = 0$. Otherwise, $G$ is not $(k+1)$-sparse. Hence there is some edge $S$ of weight at least $k+1$ with another edge of weight at least $k+1$ intersecting $S$.

As before, let $w_0$, $w_1$, $w_2$, and $w_3$ represent the total weight of edges whose intersection with $S$ has size $0$, $1$, $2$, and $3$ respectively. Let $w(S) = k + 1 + j$ for some $0 \leq j \leq 2k$, so $w_3 = k + 1 +j$. We also know $w_0 \leq \ex_{\mathbb{N}}(n-3, F_5, 3k+2) = T(n-3) + h(n-3)$. Let $H$ be the set of edges of weight at least $k - j + 1$ whose intersection with $S$ has size 2. Suppose $H$ is not empty. Then we have the following bounds from Lemma~\ref{lem:3k+1rec}:
\[
w_1 \leq 3k \binom{n-3}{2} - \sum_{e \in H} \frac{(n-4)}{4} i_e
\] and 
\[
w_2 \leq 3(k-j) (n-3) + \sum_{e \in H} i_e.
\]
Since $H$ is nonempty by assumption and $n \geq 12$, we have that 
\[
\sum_{e \in H} i_e - \sum_{e \in H} \frac{(n-4)}{4} i_e \leq -|H| \leq -1.
\]
Hence the above bounds give us
\begin{align*}
\ex_{\mathbb{N}}(n, F_5, 3k+2) &\leq w_3 + w_2 + w_1 + w_0 \\
    &\leq (k+1+j) + \left(3(k-j) (n-3) + \sum_{e \in H} i_e\right)  \\
    &\quad\quad\quad + \left(3k \binom{n-3}{2} - \sum_{e \in H} \frac{(n-4)}{4} i_e \right) + \ex_{\mathbb{N}}(n-3, F_5, 3k+2)\\
    &\leq 3k \binom{n-3}{2} + 3(k-j) (n-3) - 1 + k + 1 + j + \ex_{\mathbb{N}}(n-3, F_5, 3k+2)\\
    &= 3k \binom{n-3}{2} + 3kn -8k + 1 - 1 + 10j - 3nj + \ex_{\mathbb{N}}(n-3, F_5, 3k+2)\\
    &= k \binom{n-1}{2} + k\binom{n-2}{2} + k\binom{n-3}{2} + 1 - 1 + j(10 - 3n)\\
    &\quad\quad\quad + \ex_{\mathbb{N}}(n-3, F_5, 3k+2)\\
    &\leq k \binom{n-1}{2} + k\binom{n-2}{2} + k\binom{n-3}{2} + 1 - 1 + \ex_{\mathbb{N}}(n-3, F_5, 3k+2)\\
    &= \left(k \binom{n-1}{2} + k\binom{n-2}{2} + k\binom{n-3}{2} + 1 + T(n-3) \right)\\
    &\quad\quad\quad + h(n-3) - 1\\
    &= T(n) + h(n-3) - 1
\end{align*}
Hence $h(n) \leq h(n-3) - 1$ in this case. In the other case, we assume $H$ is empty. This means
\[
w_2 \leq 3(k-j) (n-3).
\]
By choice of $S$, there is some edge of weight at least $k+1$ that intersects with $S$. Since it is not in $H$, it must be an edge whose intersection with $S$ is size 1. Since this edge exists, we can apply our original upper bound from Theorem~\ref{thm:all3asym} for $w_1$ to get 
\[
w_1 \leq 3k\binom{n-3}{2} - \sum_{\ell = k+1}^{2k} \left( \sum_{x_i \in S}  (\ell-k) d_\ell(x_i) \right) \leq 3k\binom{n-3}{2} - 1.
\]
By essentially the same argument as the last case, we get
\begin{align*}
\ex_{\mathbb{N}}(n, F_5, 3k+2) &\leq w_3 + w_2 + w_1 + w_0 \\
    &\leq (k+1+j) + \left(3(k-j) (n-3)\right) + \left(3k\binom{n-3}{2} - 1\right)\\
    &\quad\quad\quad + \ex_{\mathbb{N}}(n-3, F_5, 3k+2)\\
    &\leq 3k \binom{n-3}{2} + 3(k-j) (n-3) - 1 + k + 1 + j + \ex_{\mathbb{N}}(n-3, F_5, 3k+2)\\
    &\vdots\\
    &\leq T(n) + h(n-3) - 1
\end{align*}
Hence $h(n) \leq h(n-3) - 1$ in this case as well. Thus either $h(n) = 0$, or $h(n) \leq h(n - 3) - 1$. 
\end{proof}
For large enough $n$, this means $h(n)$ eventually becomes $0$. This yields Theorem~\ref{thm:exactAll}.\ref{thm:exact3k+2} as a corollary.

\subsection{Proof of Theorem~\ref{thm:exactAll}.\ref{thm:exact3k+3}}
We now prove the exact result for the $3k+3$ case. This uses a similar argument to the $3k+1$ case. 

\begin{proof}
Let $n$ be large. Let $G$ be a 3-graph on $n$ vertices with edge weights $w$ that is $(F_5, 3k+3)$-free of maximum total weight. Given a tripartition $V = V_1 \cup V_2 \cup V_3$, we call an edge a \textit{cross edge} if it has one point in each $V_i$. Let $U_1 \cup U_2 \cup U_3$ be the tripartition of $V$ that maximizes the number of cross edges of weight at least $k+1$. By applying Theorem~\ref{thm:stabAll}.\ref{thm:Stab3k+3} to $\varepsilon = \frac{1}{4000000(2k+2)(k+1)^2}$, we confirm $U_1 \cup U_2 \cup U_3$ has at most $\varepsilon n^3$ edges of weight at least $k+1$ that are not cross edges. Let $H$ be a weighted $3$-graph where all cross edges have weight $k+1$ and all non-cross edges have weight $k$. Recall our lower-bound construction is such a 3-graph with a balanced tripartition, which had weight $k \binom{n}{3} + s(n)$. Consequently $w(H) \leq k \binom{n}{3} + s(n) \leq w(G)$. For the sake of contradiction, suppose $G$ has weight strictly greater than $H$. Let $t_1$ be the number of edges of weight $k+1$, and $t_2$ the number of edges of weight at least $k+2$. By Theorem~\ref{thm:stabAll}.\ref{thm:Stab3k+3i}, $t_2 \leq \varepsilon n^3$. Then
\begin{align*}
k \binom{n}{3} + s(n) \leq w(G) &\leq k \left( \binom{n}{3} - t_1 - t_2 \right) + (k+1) t_1 + (3k+2) t_2\\
&\leq k \binom{n}{3} + t_1 + (2k + 2)t_2.
\end{align*}
There are at most $\varepsilon n^3$ edges of weight at least $k+1$ that are not cross edges, so $t_1 + t_2 \leq |U_1| |U_2| |U_3| + \varepsilon n^3$. Hence 
\[
|U_1| |U_2| |U_3| + \varepsilon n^3 \geq t_1 + t_2 \geq s(n) - (2k+1) t_2 \geq s(n) - (2k+1)\varepsilon n^3.
\]
Therefore $|U_1| |U_2| |U_3| \geq s(n) - (2k+2)\varepsilon n^3$. We now consider the size of each part. Without loss of generality, suppose $U_1$ is the smallest part with size $(\frac{1}{3} - \gamma)n$. For the sake of contradiction, suppose $\gamma > \sqrt{4 (2k+3) \varepsilon}$. Note that $|U_1| + |U_2| + |U_3| = n$ such that $0 \leq |U_i| \leq n$. If $|U_1| = (\frac{1}{3} - \gamma)n$, the product is optimized by setting $|U_2| = |U_3| = (\frac{1}{3} + \frac{\gamma}{2}) n$. We can then upper bound the number of edges in $H$ as
\begin{align*}
|U_1| |U_2| |U_3| &\leq \left[ \left(\frac{1}{3} - \gamma\right) \left(\frac{1}{3} + \frac{\gamma}{2} \right)^2\right] n^3 = \left[ \frac{1}{27} - \frac{\gamma^2}{4} - \frac{\gamma^3}{4} \right] n^3 \\
&\leq \left[ \frac{1}{27} - \frac{\gamma^2}{4}  \right] n^3 < \left[ \frac{1}{27} - (2k+3) \varepsilon  \right] n^3 < s(n) - (2k+2) \varepsilon n^3.
\end{align*} 
This is a contradiction, so we must have $\gamma \leq \sqrt{4 (2k+3) \varepsilon}$. In particular, this means each part has size at least $\frac{n}{4}$. 

Let
\begin{align*}
    B_1 &= \{e \in \binom{V}{3} : e \text{ is a cross edge}, w(e) \geq k+2\}\\
    B_2 &= \{e \in \binom{V}{3} : e \text{ is not a cross edge}, w(e) \geq k+1\}\\
    M_1 &= \{e \in \binom{V}{3} : e \text{ is a cross edge}, w(e) \leq k\}\\
    M_2 &= \{e \in \binom{V}{3} : e \text{ is not a cross edge}, w(e) \leq k-1\}\\
\end{align*}
Then 
\[
    w(H) < w(G) \leq w(H) + (2k+1) |B_1| + (2k+2) |B_2| - |M_1| - |M_2|.
\]
This implies
\[
(2k+2) (|B_1| + |B_2|) > |M_1| + |M_2|
\]
We must have $B_1$ or $B_2$ nonempty for this inequality to hold. Given a vertex $v$, let $d_{B_i}(v)$ denote the number of edges of $B_i$ that are incident to $v$ and likewise define $d_{M_i}(v)$. Suppose $e = xyz \in B_1$ such that $x \in U_1$, $y \in U_2$, and $z \in U_3$. Pick any $a \in U_2 \setminus \{y\}$, $b \in U_3 \setminus \{z\}$. At least one of $abx \in M_1$, $aby \in M_2$, or $abz \in M_2$ must hold. Otherwise, $abxyz$ would form an $F_5$ of weight $3k+3$. Hence for each $e \in B_1$, there is some vertex $v \in e$ such that $d_{M_1}(v) + d_{M_2}(v) \geq \frac{1}{3}\min_{i} (|U_i| - 1)^2$. Recall that $\min_i |U_i| \geq \frac{n}{4}$. Hence $d_{M_1}(v) + d_{M_2}(v) \geq \frac{1}{3}(\frac{n}{4} - 1)^2.$

Similarly, consider any $e = xyz \in B_2$. In this case, $e$ is not a cross edge. Without loss of generality, suppose $x,y \in U_1$ while $z$ in any part. Then for any $a \in U_2 \setminus \{z\}$ and $b \in U_3 \setminus \{z\}$, at least one of $xab$ or $yab$ belong to $M_1$. Thus for any $e \in B_2$, there is some vertex $v \in e$ such that $d_{M_1}(v) + d_{M_2}(v) \geq \frac{1}{3} \min_{i} (|U_i| - 1)^2 \geq \frac{1}{3} (\frac{n}{4}-1)^2$ as well.

Let $A = \{v \in V(G) : d_{M_1}(v) + d_{M_2}(v) \geq \frac{1}{3} (\frac{n}{4} - 1)^2\}$. Then we have shown every edge $e \in B_1 \cup B_2$ has a vertex $v \in A$. Thus
\begin{align*}
\sum_{v \in A} (d_{B_1}(v) + d_{B_2}(v)) &\geq |B_1| + |B_2| \\
&\geq \frac{|M_1| + |M_2|}{2k+2}= \frac{1}{6k+6} \sum_{v \in V} (d_{M_1}(v) + d_{M_2}(v))\\
&\geq \frac{1}{6k+6} \sum_{v \in A} (d_{M_1}(v) + d_{M_2}(v)) 
\\
&\geq \frac{|A|(\frac{n}{4} - 1)^2}{18k+18}.
\end{align*}
By averaging, there is some vertex $u \in A$ such that $d_{B_1}(u) + d_{B_2}(u) \geq \frac{(\frac{n}{4} - 1)^2}{18k+18}$. Without loss of generality, suppose $u \in U_1$. Let $C_1$ be the set of $e \in B_1$ containing $u$. Let $C_2$ be the set of $e = uxy \in B_2$ containing $u$ such that $x$ and $y$ are both in the same part $U_i$. Lastly, let $C_3$ be the set of $e = uxy \in B_2$ containing $u$ such that $x$ and $y$ are in different parts (and thus one of them lies in $U_1$ with $u$). Then $|C_1| + |C_2| + |C_3| = d_{B_1}(u) + d_{B_2}(u)$. 

There are two cases to consider. First, suppose $|C_3| \leq (d_{B_1}(u) + d_{B_2}(u))/2$, so $|C_1| + |C_2| \geq (d_{B_1}(u) + d_{B_2}(u))/2$. Let $e = uxy \in C_1$. Since $e \in B_1$, suppose $x \in U_2$ and $y \in U_3$. Then for any $a \in U_1 \setminus \{u\}$ and $b \in U_3 \setminus \{y\}$, either $axb \in M_1$ or $ayb \in M_2$. This means there are at least $(\frac{n}{4}-1)^2$ edges of $M_1$ or $M_2$ of the form $xab$ or $yab$.

Similarly, consider any $e = xyu \in C_2$. In this case, $x$ and $y$ are in the same part $U_i$. If $x,y \in U_1$ along with $u$, then for any $a \in U_2$ and $b \in U_3$, one of $xab$ or $yab$ belongs to $M_1$. If $x,y \in U_2$, then for any $a \in U_1 \setminus \{u\}$ and $b \in U_{3}$, one of $axb$ or $ayb$ belongs to $M_1$. In the last case, $x,y \in U_3$. Then for any $a \in U_1 \setminus \{u\}$ and $b \in U_2$, one of $abx$ or $aby$ belongs to $M_1$. In all cases, there are at least $(\frac{n}{4}-1)^2$ edges of $M_1$ of the form $xab$ or $yab$.

Note for both $C_1$ and $C_2$, if we count the edges that we found of the form $xab$ or $yab$ in $M_1 \cup M_2$, we overcount each edge at most $3n$ times. Hence 
\begin{align*}
|M_1| + |M_2| &\geq \frac{1}{3n}\sum_{e \in C_1 \cup C_2} \left(\frac{n}{4} - 1\right)^2 \geq \frac{1}{3n} \frac{(\frac{n}{4} - 1)^2}{36k+36} \left(\frac{n}{4} - 1\right)^2\\
&= \frac{n^3}{27648(k+1)} + O(n^2) > \frac{n^3}{2000000(k+1)^2} \\
&= 2(2k+2) \frac{n^3}{4000000(2k+2)(k+1)^2} = 2(2k+2) \varepsilon n^3.
\end{align*}

In the other case, we suppose $|C_3| \geq (d_{B_1}(u) + d_{B_2}(u))/2$. Let $C_3^{(i)}$ be the set of edges $e = uxy$ where $x \in U_1$ with $u$ and $y \in U_i$ (for $i \neq 1$). Let $L$ be the set of cross edges including $u$ that have weight at least $k+1$. Recall $U_1 \cup U_2 \cup U_3$ is the tripartition of $V$ that maximizes the number of cross edges of weight at least $k+1$. For the sake of contradiction, suppose $|C_3^{(i)}| > |L|$ for some $i$. Without loss of generality, let $i = 2$. Then consider moving $u$ to $U_3$. Then every edge in $C_3^{(i)}$ becomes a cross edge, while every edge of $L$ becomes a non-cross edge.  This contradicts the maximality of $U_1 \cup U_2 \cup U_3$. Thus $|C_3^{(i)}| \leq |L|$ for each $i$, so $|C_3| \leq 2 |L|$.

Now suppose $e = uxy \in C_3$. Without loss of generality, suppose $u, x \in U_1$ and $y \in U_2$. For any $a \in U_2 \setminus \{y\}$ and $b \in U_3$, one of $uab$ or $xab$ belongs to $M_1$. However, if $uab \in L$, then we must have $xab \in M_1$. This means there are at least 
\[
|L| - n \geq (d_{B_1}(u) + d_{B_2}(u))/4 - n \geq \frac{(\frac{n}{4} - 1)^2}{72k+72} - n
\]edges in $M_1$ of the form $xab$. As before, we overcount each edge at most $3n$ times. Hence 
\begin{align*}
|M_1| + |M_2| &\geq \frac{1}{3n}\sum_{e \in C_3} \left( \frac{(\frac{n}{4} - 1)^2}{72k+72} - n \right) \geq \frac{1}{3n} \frac{(\frac{n}{4} - 1)^2}{36k+36} \left( \frac{(\frac{n}{4} - 1)^2}{72k+72} - n \right)\\
&= \frac{n^3}{1990656(k+1)^2} + O(n^2) > \frac{n^3}{2000000(k+1)^2} \\
&= 2(2k+2) \frac{n^3}{4000000(2k+2)(k+1)^2} = 2(2k+2) \varepsilon n^3.
\end{align*}

Therefore, in both cases $|B_1| + |B_2| \geq \frac{|M_1| + |M_2|}{2k+2} > 2 \varepsilon n^3$. By the pigeonhole principle, either $|B_1| > \varepsilon n^3$ or $|B_2| > \varepsilon n^3$. This contradicts Theorem~\ref{thm:stabAll}.\ref{thm:Stab3k+3i} and Theorem~\ref{thm:stabAll}.\ref{thm:Stab3k+3ii} respectively. Hence $w(G) \leq w(H)$. Since $H$ is a 3-graph where all cross edges have weight $k+1$ and all non-cross edges have weight $k$, $w(G) \leq w(H) \leq k \binom{n}{3} + s(n)$, which is indeed achievable by our lower bound construction. Hence $w(G) = k \binom{n}{3} + s(n)$.
\end{proof}

\bibliographystyle{plainurl}
\bibliography{refs}{}
\end{document}